\documentclass[english]{amsart}
\usepackage{amsfonts, amsmath, amssymb, amscd, amsthm, array, graphicx}
\usepackage[latin9]{inputenc}
\usepackage{amsmath}
\usepackage{mathabx}
\usepackage{stmaryrd}
\usepackage{graphicx}
\usepackage{amssymb}
\usepackage{dsfont}
\usepackage{babel}
\usepackage{color}
\usepackage[all]{xy}

\usepackage{mathtools}
\usepackage{hyperref,url}

\makeatletter
\def\blfootnote{\xdef\@thefnmark{}\@footnotetext}
\makeatother
\usepackage{tikz}
\usetikzlibrary{arrows,shapes,snakes,automata,backgrounds,petri}
\usetikzlibrary{automata} 
\usetikzlibrary[automata] 
\usetikzlibrary{matrix,decorations.pathreplacing}
\usepackage{geometry}
\usepackage{pdflscape}
\usepackage{graphicx}
\usetikzlibrary{external,automata,trees,positioning,shadows,arrows,shapes.geometric}

\newlength\Textht
\newtheorem{thm}{Theorem}[section]
\newtheorem*{thm*}{Theorem}
\newtheorem{cor}[thm]{Corollary}
\newtheorem{lem}[thm]{Lemma}
\newtheorem{prop}[thm]{Proposition}

\theoremstyle{definition}

\newtheorem{rem}[thm]{Remark}

\newcommand{\id}{\operatorname{Id}}
\newcommand\ZZ{\mathbb{Z}}

\newcommand{\QQ}{{\mathbb Q}}

\newcommand{\Mat}{\operatorname{M}}
\newcommand{\GL}{\operatorname{GL}}

\newcommand{\Aut}{\operatorname{Aut}}

\newcommand{\End}{\operatorname{End}}
\newcommand{\Inn}{\operatorname{Inn}}

\newcommand{\gen}[1]{\left\langle#1\right\rangle}

\newcommand{\pres}[2]{\left\langle#1 \mid #2\right\rangle}
\newcommand{\ncl}[1]{\ll \!#1 \!\gg}

\newcommand{\R}[1]{{\color{red} #1}}

\newcounter{mallikacomments}

\begin{document}

\title{Twisted conjugacy in $BS(n, 1)$}

\author{Oorna Mitra}
\address{Indian Institute of Science Education and Research (IISER) Kolkata, Mohanpur, Nadia - 741246, West Bengal, India.}
\email{urna.mitra@gmail.com}

\author{Mallika Roy}
\address{Harish-Chandra Research Institute, HBNI, Chhatnag Road, Jhunsi, Prayagraj (Allahabad) 211019, India.} \email{mallikaroy@hri.res.in}

\author{Enric Ventura}
\address{Departament de Matem\`atiques, Universitat Polit\`ecnica de Catalunya, CATAlONIA.} \email{enric.ventura@upc.edu}



\begin{abstract}
In this article, we solve the twisted conjugacy problem for solvable Baumslag--Solitar groups $BS(n,1)$, i.e., we propose an algorithm which, given two elements $u,v \in BS(n,1)$ and an automorphism $\varphi \in \Aut(BS(n,1))$, decides whether $v=(w\varphi)^{-1} u w$ for some $w\in BS(n,1)$. Also we prove that the automorphism group $\Aut(BS(n,1))$ is orbit decidable --- given two words on the generators $u,v\in F(X)$, decide whether the corresponding elements $u,v\in G$ can be mapped to each other by some automorphism in $\Aut(BS(n,1))$.
\end{abstract}

\maketitle



\section{Introduction}

The study of Algorithmic Group Theory started with three decision problems, namely the \textit{Word Problem}, the \textit{Conjugacy Problem} and the \textit{Isomorphism Problem}, which are also known as the three \textit{Dehn Problems} (see~\cite{Dehn}): 

\noindent \textbf{\boldmath Word Problem, $WP(G)$:} For a finite presentation $G=\pres{X}{R}$, \emph{given a word on the generators $w\in F(X)$, decide whether $w$ represents the trivial element in $G$, $w=_G 1$.}

\noindent \textbf{\boldmath Conjugacy Problem, $CP(G)$:} For a finite presentation $G=\pres{X}{R}$, \emph{given two words on the generators $u,v\in F(X)$, decide whether $u$ and $v$ represent conjugate elements in $G$, $u\sim_G v$.}

\noindent \textbf{\boldmath Isomorphism Problem, $IP$:} \emph{Given two finite presentations, $\pres{X}{R}$ and $\pres{Y}{S}$, decide whether they present isomorphic groups, $\pres{X}{R}\simeq \pres{Y}{S}$.}

As it is well known, these three problems are unsolvable in their full generality, and there are lots of interesting results in the literature solving them on certain families of groups, or analyzing the solvability boundary on some others. In the present article we will focus on the conjugacy problem in close relation with another two algorithmic problems: the \textit{Twisted Conjugacy Problem} and the \textit{Orbit Decidability Problem}.

Let $G$ be a group and $\varphi \in \Aut(G)$, $g\mapsto g\varphi$, be an automorphism. One says that two elements $u,v \in G$ are \emph{$\varphi$-twisted conjugated}, denoted $u\sim_\varphi v$, if there exists $x\in G$ such that $v=(x\varphi )^{-1}ux$. It is easy to see that $\sim_\varphi$ is an equivalence relation on $G$, which coincides with the usual conjugation when $\varphi=\id$. This variation of standard conjugacy was first introduced by Reidemeister with clear topological motivations related to Nielsen Fixed Point Theory; see~\cite{R}. The corresponding algorithmic problem is the following.

\noindent \textbf{\boldmath Twisted Conjugacy Problem, $TCP(G)$:} For a finite presentation $G=\pres{X}{R}$, \emph{given an automorphism $\varphi\colon G\to G$ (given by the images of the generators), and two words on the generators, $u,v\in F(X)$, decide whether $u$ and $v$ represent $\varphi$-twisted conjugate elements in $G$, i.e., $u\sim_{\varphi} v$.}

Of course, $CP$ coincides with the particular case of $TCP$ with $\varphi=\id$; so, a positive solution to $TCP(G)$ immediately gives a positive solution to $CP(G)$ (which automatically gives a positive solution to $WP(G)$ as well). However, as one might expect, in an arbitrary group $G$, twisted conjugacy classes are much more complicated to understand than standard conjugacy classes. For instance, in the case of free groups, $CP(F_n)$ is very easy both conceptually and computationally, while $TCP(F_n)$ is solvable but much harder in both senses; see~\cite[Theorem 1.5]{BMMV}. Also, similar to the classical result stating the existence of a finitely presented group $G$ with solvable $WP(G)$ but unsolvable $CP(G)$ (see, for example, \cite{M}), there is a more recent result stating the existence of a finitely presented group $G$ with solvable $CP(G)$ but unsolvable $TCP(G)$; see~\cite[Corollary 4.9]{BMV}.

In~\cite{BMMV} Bogopolski--Martino--Maslakova--Ventura solved the conjugacy problem for free-by-cyclic groups, $CP(F_n\rtimes_{\varphi} \ZZ)$. The proof had two clearly separate ingredients: firstly a solution to the Twisted Conjugacy Problem for free groups, $TCP(F_n)$, and, secondly, a result from Brinkmann~\cite{B} providing an algorithm to decide, given an automorphism $\varphi\in \Aut(F_n)$, and two elements $u,v\in F_n$, whether $v$ is conjugate to some iterated image of $u$, i.e.,  $v\sim u\varphi^k$, for some $k\in \ZZ$. On the other hand, more than 30 years earlier, Miller~\cite{M} already studied the much bigger family of free-by-free groups, and showed that the Conjugacy Problem is already unsolvable within that framework. 

A few years later, Bogopolski--Martino--Ventura~\cite{BMV} realized that, in their solution to $CP(F_n\rtimes_{\varphi} \ZZ)$, the case where the two inputs belonged to outside the normal subgroup $F_n$ ($u,v\not\in F_n$) boils down to $TCP(F_n)$, while the opposite case ($u,v\in F_n$) corresponded directly to Brinkmann's result (note that if $u\in F_n$ and $v\not\in F_n$, or viceversa, then they automatically are not conjugate to each other). This pattern happened to be much more general: in the free-by-free group $G=F_n\rtimes_{\varphi_1,\ldots ,\varphi_m} F_m$, the decision on whether two elements $u,v\not\in F_n$ are conjugated to each other boils down to $TCP(F_n)$ (so, for every free-by-free group, the $CP$ \emph{is} solvable \emph{when restricted} to inputs outside $F_n$), while deciding whether two elements $u,v\in F_n$ are conjugated to each other in $G$, becomes a much harder problem than Brinkmann's one (and mandatorily sometimes unsolvable, according to Miller's result). Much beyond this family, the main result in~\cite{BMV} was the following one, exactly in the same spirit, but applying to \emph{any} short exact sequence of groups:

\begin{thm}[Bogopolski--Martino--Ventura~\cite{BMV}] \label{thm: TCP to CP}
Let $1 \xrightarrow{\hspace{0.2 cm}} F \xrightarrow{\hspace{0.2 cm}\alpha} G \xrightarrow{\hspace{0.2 cm}\beta} H \xrightarrow{\hspace{0.2 cm}} 1$
be an algorithmic short exact sequence of groups such that
 \begin{itemize}
\item[(i)] $F$ has solvable twisted conjugacy problem,
\item[(ii)] $H$ has solvable conjugacy problem, and
\item[(iii)] for every $1\neq h\in H$, the subgroup $\langle h\rangle$ has finite index in its centralizer $C_H(h)$, and there is an algorithm which computes a finite set of coset representatives, $z_{h,1}, \ldots ,z_{h,t_h} \in H$, $C_H(h) =\langle h\rangle z_{h,1} \sqcup \ldots \sqcup \langle h\rangle z_{h,t_h}$.
 \end{itemize}
Then, the following are equivalent:
 \begin{itemize}
\item [(a)] the conjugacy problem for $G$ is solvable,
\item [(b)] the conjugacy problem for $G$ restricted to $F$ is solvable,
\item [(c)] the action subgroup $A_G = \{\varphi_g \, | \, g \in G \} \leqslant \Aut(F)$ is orbit decidable.
 \end{itemize}
\end{thm}

The term \emph{algorithmic} referred to a short exact sequence as above, just means the assumption of algorithms for computing images of elements by $\alpha$ and $\beta$, pre-images by $\beta$, and pre-images by $\alpha$ of elements in $\ker \beta =\operatorname{Im} \alpha$; typically, this is the case when $F$, $G$ and $H$ are given by finite presentations, and $\alpha$ and $\beta$ are given by images of the generators; see~\cite{BMV} for details. As for the action subgroup of the short exact sequence, $A_G$, this is just the set of restrictions to $F\unlhd G$, namely $\varphi_g \colon F\to F$, of the inner automorphisms of $G$, namely $\gamma_g \colon G\to G$, $x\mapsto g^{-1}xg$ (\,note that $\varphi_g$ is not, in general, inner as automorphism of $F$; we can only say that $A_G\leqslant \Aut(F)$\,). Finally, the notion of \emph{Orbit Decidability} is crucial: 

\noindent \textbf{\boldmath Orbit Decidability Problem, $OD(A)$:} For a finite presentation $G=\pres{X}{R}$ and a group of automorphisms $A\leqslant \Aut(G)$, \emph{given two words on the generators $u,v\in F(X)$, decide whether the corresponding elements $u,v\in G$ can be mapped to each other by some automorphism in $A$, i.e., whether there exists $\alpha \in A$ such that $u\alpha =v$.}

\begin{rem}
Sometimes the notion of Orbit Decidability for $A\leqslant \Aut(G)$ asks to decide whether there exists $\alpha \in A$ such that $u\alpha \sim v$. In general, this is not exactly the same, but the two notions obviously coincide when $\Inn(G)\leqslant A$, as is always the case for the action subgroups coming from short exact sequences. Therefore, Theorem~\ref{thm: TCP to CP} is perfectly valid with the two (coinciding) notions of orbit decidabillity. \end{rem}

A subgroup $A\leqslant \Aut(G)$ is said to be \emph{orbit decidable} if $OD(A)$ is solvable. Particularizing to the case where both $F$ and $H$ are free groups (hypotheses (i)-(ii)-(iii) do hold), the group $G$ siting in the middle of the short exact sequence is free-by-free, and Theorem~\ref{thm: TCP to CP} states that $CP(G)$ is solvable if and only if the corresponding action subgroup is orbit decidable: sometimes this is the case (for example, when $H=\ZZ$, in which case $A_G/(A_G\cap \Inn(F))$ is cyclic and $OD(A_G)$ is, precisely, Brinkmann's problem), and sometimes it is not (like in the classical Miller's examples of free-by-free groups with unsolvable Conjugacy Problem); see Bogopolski--Martino--Ventura~\cite{BMV} for more details. 

After the publication of~\cite{BMV}, several research papers have appeared following the same strategy in other situations: anytime we have a family of groups (to put in place of $F$) for which we can solve the Twisted Conjugacy Problem, it becomes interesting to study Orbit Decidability for (all, certain) subgroups of automorphisms $A\leqslant \Aut(F)$, and deduce from it the solvability or unsolvability of the Conjugacy Problem for extensions of $F$ by families of groups $H$ satisfying (ii) and (iii) (for example, torsion-free hyperbolic groups; see~\cite{BMV}). This strategy was followed in~\cite{BMV} itself to study the Conjugacy Problem within the family of free-by-free-abelian groups, where the authors positively solved it for all groups of the form $\ZZ^2 \rtimes F_m$, and $\ZZ^n \rtimes_A F_m$ with virtually solvable action subgroup $A\leqslant \GL_n(\ZZ)$; and provided the first known examples of groups of the form $\ZZ^4 \rtimes F_n$ with unsolvable Conjugacy Problem. Similarly, Sunic--Ventura \cite{SV} constructed the first known examples of automaton groups with unsolvable Conjugacy Problem. Gonz\'alez-Meneses--Ventura \cite{GM-V} followed the same project for Braid groups, showing that all Braid-by-hyperbolic groups have solvable Conjugacy Problem. Burillo--Matucci--Ventura \cite{BuMV} did the same for Thompson's group $F$, proving in this case the existence of Thomson-by-hyperbolic groups with unsolvable Conjugacy Problem. Recently, Blufstein--Valiunas~\cite{BV} and Crowe~\cite{Cr, Cr2} have followed similar projects for certain large-type Artin groups and for dihedral Artin groups, respectively. 

In the present paper we follow the above project for the family of solvable Baumslag--Solitar groups $BS(n,1)$: we first solve the Twisted Conjugacy Problem, and then investigate Orbit Decidability of certain subgroups of $\Aut(BS(n,1))$, connecting it with the Conjugacy Problem for some extensions of $BS(n,1)$.

At the beginning of Section~\ref{sec: BS overview} we briefly survey elementary properties and a normal form for the elements and the semi-direct product structure of $BS(n,1)$ (see~\ref{sec: semi direct}). Then we move on to the automorphism group $\Aut(BS(n,1))$: in~\ref{sec: autos}, following the results of Collins--Levin~\cite{CoLe} and O'Neil~\cite{N}, we give a more direct proof regarding the explicit form for the automorphisms of $BS(n,1)$. In~\ref{sec: CP} we give an explicit solution to the Conjugacy Problem for $BS(n,1)$. Then we focus on and solve the Twisted Conjugacy Problem (see Theorem~\ref{thm: main TCP}) in Section~\ref{sec: TCP}, and the orbit decidability (see Theorem~\ref{thm: main 1 OD}) for the full automorphism group $\Aut(BS(n,1))$ in Section~\ref{sec: OD}. We conclude the solvability of the Conjugacy Problem in the corresponding extension groups. 



\subsection*{General notation and conventions}

For a group $G$, $\Aut(G)$ (resp., $\End(G)$) denotes the group (resp., the monoid) of automorphisms (resp., endomorphisms) of $G$. We write them all with the argument on the left, that is, we denote by $(x)\varphi$ (or simply $x\varphi$) the image of the element $x$ by the homomorphism $\varphi$; accordingly, we denote by $\varphi \psi$ the composition $A \xrightarrow{\varphi} B \xrightarrow{\psi} C$. Specifically, we will reserve the letter $\gamma$ for right conjugations, $\gamma_g \colon G \to G, x\mapsto g^{-1}xg$. Following the same convention above, when thinking of a matrix $A$ as a map, it will always act on the right of horizontal vectors, $\textbf{v}\mapsto \textbf{v}A$. We denote by $\Mat_{n\times m}(\ZZ)$ the $n \times m$ (additive) group of matrices over $\ZZ$, and by $\GL_m(\ZZ)$ the linear group over the integers. We use the function $| \cdot |_a$ to count the number of $a$ occurrences in a word $w$.

\section{Solvable Baumslag--Solitar groups}\label{sec: BS overview}

The so-called Baumslag--Solitar groups $BS(n,m)$ are the class of two-generated one-related groups presented by 
 $$
BS(n,m) = \pres{a,t}{t^{-1}a^n t=a^m},
 $$
for $n, m\in \ZZ$. They all are HNN-extensions of $\ZZ$ (with associated subgroups $n\ZZ$ and $m\ZZ$) sitting in the middle of the well known short exact sequence
 \begin{equation}\label{ses}
\begin{array}{ccccccccc} 1 & \to\, & \ncl{a} & \, \to & BS(n,m) & \stackrel{\pi}{\to} & \ZZ =\gen{t} & \to & 1. \\ & & & & a & \mapsto & 1 & & \\ & & & & t & \mapsto & t & &
\end{array}
 \end{equation}
Note that the homomorphism $\pi\colon BS(n,m)\twoheadrightarrow \ZZ$ counts the total $t$-exponent in words $w=w(a,t)$, namely $w\mapsto t^{|w|_t}$, which is a well defined invariant because the defining relation is $t$-balanced (not being the case for $a$, in general). Of course, $\ker \pi =\ncl{a}\, =\{w(a,t) \mid |w|_t =0\}\unlhd BS(n,m)$.

In this paper we will concentrate in the solvable groups within this family, namely those with $m=1$, $BS(n,1)$. As the first elementary (and special) examples, we have $BS(0,1)\simeq \ZZ$, $BS(1,1)\simeq \ZZ^2$, or $BS(-1,1)=\pres{a,t}{t^{-1}a^{-1}t=a}$, which is the fundamental group of the Klein bottle. Note that $BS(n,1)^{\rm ab}=\ZZ \oplus \ZZ/|n-1|\ZZ$, with $\ncl{a}$ being, for $n\neq 1$, the only normal subgroup whose quotient is $\ZZ$; in particular, it is invariant under any automorphism, a crucial property for the arguments below.

\subsection{Normal Form}

Looking at the defining relation under the forms $a^n t=ta$ and $a^{-n}t=ta^{-1}$, one observes that, in an arbitrary word $w(a,t)$, positive occurrences of $t$ can always be moved to the right, at the price of multiplying by $n$ the exponents of the jumped $a$'s. Similarly, the defining relation written in the forms $t^{-1}a^n =at^{-1}$ and $t^{-1}a^{-n}=a^{-1}t^{-1}$ tells us that negative occurrences of $t$'s can always be moved to the left, at the price of multiplying by $n$ the exponents of the jumped $a$'s. Summarizing (and using induction), for every $k,p\in \ZZ$, $p\geq 0$, we have 
 \begin{equation}\label{eq: jumps}
t^p a^k =a^{kn^p}t^p \quad \text{ and } \quad a^k t^{-p}=t^{-p}a^{kn^p}.
 \end{equation}
Repeating sufficiently many of these jumps, any element $w(a,t)\in BS(n,1)$ can be expressed in the form $g=t^{-p}a^k t^q$, where $k,p,q\in \ZZ$, $p,q\geq 0$. This is a classical result which, from now on, for technical reasons, we prefer to write in the form $g=t^{-p}a^k t^p t^c$, where $p,k, c\in \ZZ$, $p\geq 0$. This is close to be a normal form, as will be proven below.

%

To understand the structure of $BS(n,1)$, we consider the following monomorphism, which proves linearity for this family of groups. 

\begin{lem}\label{linearity}
The map
 \begin{equation}\label{eq: BS into GL2}
\begin{array}{rcl} \varphi\colon BS(n,1) & \to & \GL_2(\ZZ[1/n])\leqslant \GL_2(\QQ) \\[3pt] a & \mapsto & A=\left(\begin{smallmatrix} 1 & 1 \\ 0 & 1 \end{smallmatrix}\right) \\[3pt] t & \mapsto & T=\left(\begin{smallmatrix} n & 0 \\ 0 & 1 \end{smallmatrix}\right) \end{array}
 \end{equation}
defines a monomorphism of groups. In particular, $BS(n,1)$ is a linear group. 
\end{lem}

\begin{proof}
It is straightforward to check that the map~\eqref{eq: BS into GL2} preserves the relation $t^{-1}a^n t=a$ and, hence, it is well defined:
 $$
T^{-1} A^n T=\begin{pmatrix} 1/n & 0 \\ 0 & 1 \end{pmatrix}
\begin{pmatrix} 1 & n \\ 0 & 1 \end{pmatrix} \begin{pmatrix} n & 0 \\ 0 & 1 \end{pmatrix}=\begin{pmatrix} 1 & 1 \\ 0 & 1 \end{pmatrix}=A.
 $$
Take an arbitrary element $g=t^{-p}a^kt^q\in BS(n,1)$. If $I=(t^{-p}a^kt^q)\varphi =T^{-p} A^k T^q$ then
 $$
\begin{pmatrix} 1 & 0 \\ 0 & 1 \end{pmatrix}= \begin{pmatrix} n^{-p} & 0 \\ 0 & 1 \end{pmatrix} \begin{pmatrix} 1 & k \\ 0 & 1 \end{pmatrix} \begin{pmatrix} n^q & 0 \\ 0 & 1 \end{pmatrix} =\begin{pmatrix} n^{q-p} & k\,n^{-p} \\ 0 & 1 \end{pmatrix}
 $$
and hence $k=0$ and $p=q$; therefore, $g=1$. This shows that $\varphi$ is injective.
\end{proof}

As a straightforward first corollary we get uniqueness, up to certain point, of the above expression for elements of $BS(n,1)$. 


\begin{cor}
Let $p_i, k_i, c_i\in \ZZ$ with $p_i\geq 0$, for $i=1,2$. In the group $BS(n,1)$, $t^{-p_1}a^{k_1} t^{p_1+c_1}=t^{-p_2}a^{k_2} t^{p_2+c_2}$ if and only if $c_1 =c_2 \in \ZZ$ and $k_1/n^{p_1}=k_2/n^{p_2}\in \ZZ[1/n]$.
\end{cor}

\begin{proof}
Applying $\pi$ from~\eqref{ses}, it is clear that $c_1 =c_2$ is a necessary condition. Then, the assumed equality can be expressed as 
 $$
1=t^{-p_1}a^{k_1} t^{p_1-p_2} a^{-k_2} t^{p_2} = t^{-p_1-p_2}a^{k_1 n^{p_2}} a^{-k_2n^{p_1}} t^{p_1+p_2},
 $$
which is equivalent to $a^{k_1 n^{p_2} -k_2n^{p_1}}=1$ and so, to $k_1 n^{p_2} -k_2n^{p_1}=0\in \ZZ$, and to $k_1/n^{p_1}=k_2/n^{p_2}\in \ZZ[1/n]$ (since, by Lemma~\ref{linearity}, $a$ has infinite order in $BS(n,1)$).
\end{proof}

This suggests to use the following notation: for $\alpha =k/n^p\in \ZZ[1/n]$ (here, $k,p\in \ZZ$, $p\geq 0$) write 
 $$
a^{\alpha}:=t^{-p}a^k t^p.
 $$
Note that this is coherent because the equality among rational numbers $kn/n^{p+1} =k/n^p$ corresponds to the equality $t^{-p-1}a^{kn} t^{p+1}=t^{-p} (t^{-1}a^n t)^k t^p=t^{-p}a^k t^p$ among the corresponding elements in $BS(n,1)$. This is specific for powers of $a$ so, we only accept rational exponents for the letter $a$. 

Moreover, the two rules~\eqref{eq: jumps} stated above and saying that $t$ jumps to the right (and $t^{-1}$ to the left) of any integral power of $a$ at the price of multiplying its exponent by $n$, 
can be unified and extended to the following more homogeneous rule, saying that $t^c$, for $c\in \ZZ$ (with no signum distinction !), jumps to the right of rational powers of $a$ at the price of multiplying its exponent by $n^c$. Furthermore, this new exponential notation is compatible with the standard rules of computation. 

\begin{lem}\label{lem: calculs}
For any $\alpha =k/n^p,\, \beta=\ell/n^q\in \ZZ[1/n]$ ($k,\ell, p,q\in \ZZ$, $p,q\geq 0$) and $c, r\in \ZZ$, we have 
\begin{itemize}
\item[(i)] $t^c a^{\alpha} =a^{n^c \alpha} t^c$ (equivalently, $a^{\alpha} t^c= t^c a^{\alpha n^{-c}}$);
\item[(ii)] $a^{\alpha}a^{\beta}=a^{\alpha+\beta}$;
\item[(iii)] $(a^{\alpha})^r =a^{r\alpha}$;
\item[(iv)] $t^{-c}a^{\alpha}t^c=a^{\alpha/n^c}$;
\item[(v)] $\big( a^{\alpha}t^c \big)^r = a^{\frac{n^{rc}-1}{n^c-1}\alpha} t^{rc}$.
\end{itemize}
\end{lem}

\begin{proof}
For $c\geq 0$ we have $t^c a^{\alpha} =t^c t^{-p}a^k t^{p}= t^{-p}a^{kn^c} t^{c}t^{p}=a^{n^c \alpha} t^c$. And, for $c\leq 0$, we have $t^c a^{\alpha} =t^c t^{-p}a^k t^{p}=t^{c-p}a^k t^{p-c}t^c= a^{k/n^{p-c}} t^{c} =a^{n^c \alpha} t^c$. In a similar fashion, we get the equivalent version $a^{\alpha} t^c= t^c a^{n^{-c} \alpha}$. This proves~(i).

For~(ii), $a^{\alpha}a^{\beta}=t^{-p}a^k t^{p-q} a^{\ell} t^q =t^{-p}a^ka^{\ell n^{p-q}} t^p =a^{(k+\ell n^{p-q})/n^p}=a^{\alpha+\beta}$, where the second equality uses~(i).

For~(iii), $(a^{\alpha})^r =\big(t^{-p}a^k t^p\big)^r =t^{-p} a^{rk} t^p=a^{rk/n^p}=a^{r\alpha}$.

For~(iv), $t^{-c}a^{\alpha}t^c=a^{n^{-c}\alpha}t^{-c}t^c =a^{\alpha/n^c}$, using~(i).

Finally, (v) is obvious for $r=0,1$. For $r\geq 2$ we use induction on $r$: 
 $$
\big( a^{\alpha} t^c \big)^{r+1}= a^{\alpha}t^c a^{\frac{n^{rc}-1}{n^c-1}\alpha} t^{rc} = a^{\alpha +n^c \frac{n^{rc}-1}{n^c-1}\alpha} t^{(r+1)c} =a^{(1+n^c \frac{n^{rc}-1}{n^c-1})\alpha} t^{(r+1)c}=
 $$
 $$
=a^{\frac{n^c-1+n^{(r+1)c}-n^c}{n^c-1}\alpha} t^{(r+1)c} =a^{\frac{n^{(r+1)c}-1}{n^c-1}\alpha} t^{(r+1)c},
$$
again using~(i). And, for $r\leq 0$, we apply the formula to $-r\geq 0$ and get  
 $$
\big( a^{\alpha} t^c \big)^r =\big( \big( a^{\alpha} t^c \big)^{-1} \big)^{-r} =\big( a^{-\alpha n^{-c}} t^{-c} \big)^{-r} =a^{\frac{n^{rc}-1}{n^{-c}-1}(-\alpha n^{-c})} t^{rc} =a^{-\frac{n^{rc}-1}{1-n^c}\alpha} t^{rc} =a^{\frac{n^{rc}-1}{n^c-1}\alpha} t^{rc}.
 $$
This completes the proof.
\end{proof}

Using this notation and the above arguments, we can conclude the following normal form for elements in $BS(n,1)$:

\begin{prop}
Every element $g\in BS(n,1)$ can be written, in a unique way, as $g=a^{\alpha}t^c$, for some $\alpha\in \ZZ[1/n]$ and some $c\in \ZZ$.\qed
\end{prop}

Another straightforward consequence of Lemma~\ref{linearity} is the structure of $\ker \pi=\ncl{a}$ of~\eqref{ses}. 

\begin{cor}\label{cor: kernel}
We have $\ker \pi =\ncl{a} \,\simeq \ZZ[1/n]$, an additive subgroup of $\QQ$. In particular, it is abelian and not finitely generated. 
\end{cor}

\begin{proof}
Clearly, 
 $$
\ker \pi =\ncl{a} =\{a^{\alpha}t^c\in BS(n,1) \mid 1=\big(a^{\alpha}t^c\big)\pi =t^c\}=\{a^{\alpha} \mid \alpha\in \ZZ[1/n]\}.
 $$
Since, by Lemma~\ref{lem: calculs}(ii),  $a^{\alpha}a^{\beta}=a^{\alpha+\beta}$, we conclude that $\ker \pi =\ncl{a} \,\simeq \ZZ[1/n]$, an additive subgroup of $\QQ$; in particular, it is abelian. Alternatively, this can also be seen by observing that the image of $\ker \pi =\ncl{a}$ under the monomorphism $\varphi$ of~\eqref{eq: BS into GL2} is $\langle \big( \begin{smallmatrix} 1 & 1/n^p \\ 0 & 1 \end{smallmatrix} \big), \, p\geq 0\rangle \leqslant \GL_2(\ZZ[1/n])$.
%
\end{proof}


\subsection{The semidirect product structure for $BS(n,1)$}\label{sec: semi direct}

According to Corollary~\ref{cor: kernel}, the short exact sequence~\eqref{ses} for $m=1$ has the form 
 $$
1\to \mathbb{Z}[1/n] \to BS(n,1) \to \mathbb{Z} \to 1,
 $$
and it splits because $\ZZ$ is free. We conclude that $BS(n, 1)$ is a semidirect product $BS(n, 1) \simeq \mathbb{Z}[1/n] \rtimes \mathbb{Z}$, where $\ZZ=\gen{t}$ acts on $\ZZ[1/n]$ by multiplying by $n$. More explicitly, 
 $$
BS(n,1)=\{ a^{\alpha}t^c \mid \alpha\in \ZZ[1/n],\, c\in \ZZ\},
 $$
with the group operation given by 
 \begin{align}\label{eq: product}
\big( a^{\alpha_1}t^{c_1} \big)\big( a^{\alpha_2}t^{c_2} \big) &= a^{\alpha_1}t^{c_1} a^{\alpha_2}t^{c_2} \notag \\ &= a^{\alpha_1} a^{n^{c_1}\alpha_2}t^{c_1+c_2} \\ &= a^{\alpha_1 +n^{c_1}\alpha_2} t^{c_1+c_2}, \notag
 \end{align}
and the inversion by 
 \begin{equation}\label{eq: inverse}
\big( a^{\alpha}t^c \big)^{-1} =t^{-c}a^{-\alpha}=a^{-n^{-c}\alpha}t^{-c}.
 \end{equation}



\subsection{Automorphisms of $BS(n,1)$}\label{sec: autos}

Following a previous work by Collins--Levin~\cite{CoLe}, J. O'Neill~\cite{N} gave an explicit description of the automorphisms of $BS(n,1)$. For completeness, we give here a more direct and compact proof. 

\begin{prop}[J. O'Neill~\cite{N}]\label{prop: O'Neil}
The automorphism group of $BS(n,1)$ is 
 $$
\Aut(BS(n,1))=\{\varphi_{\alpha, \beta} \mid \alpha\in \ZZ[1/n]^*,\, \beta \in \ZZ[1/n]\}
 $$
where, for $\alpha=k/n^p$ and $\beta=\ell/n^q$ with $k,\ell, p,q\in \ZZ$, $p,q\geq 0$, $\varphi_{\alpha, \beta}$ is defined as
 $$
\begin{array}{rcl} \varphi_{\alpha, \beta}\colon BS(n,1) & \to & BS(n,1) \\ a & \mapsto & a^{\alpha}=t^{-p} a^k t^p \\ t & \mapsto & a^{\beta}t=t^{-q}a^{\ell} t^{q+1}. 
 \end{array}
 $$
Moreover, $\varphi_{\alpha_1, \beta_1} \circ \varphi_{\alpha_2, \beta_2} =\varphi_{\alpha_1\alpha_2, \beta_1\alpha_2+\beta_2}$.
\end{prop}

\begin{proof}

It is clear that any automorphism $\varphi\in \Aut(BS(n,1))$ must be of the above form: it must leave $\ncl{a}$ invariant (because it is the only normal subgroup in $BS(n,1)$ with quotient $\ZZ$) and so, it must send $a$ to an element with $|a\varphi|_t=0$, and $t$ to an element with $|t\varphi|_t=1$; the necessity of the invertibility condition for $\alpha=k/n^p\in \ZZ[1/n]^*$ (just meaning that $k$ divides some power of $n$) will be seen later. 

Let us see that every such $\varphi_{\alpha, \beta}$ is, in fact, an endomorphism. It is straightforward to check it is well defined, by proving the preservation of the defining relation:
 \begin{align*}
t^{-1}a^nt \stackrel{\varphi_{\alpha,\beta}}{\quad \mapsto\quad} &\big( a^{\beta} t\big)^{-1} \big( a^{\alpha}\big)^n \big( a^{\beta} t\big)= t^{-1} a^{-\beta} a^{n\alpha} a^{\beta}t= t^{-1} a^{n\alpha} t =a^{\alpha},
\\ a \stackrel{\varphi_{\alpha,\beta}}{\quad \mapsto\quad} & a^{\alpha}. 
 \end{align*}

Now let us check that the composition $\varphi_{\alpha_1, \beta_1} \circ \varphi_{\alpha_2, \beta_2}$ equals $\varphi_{\alpha_1\alpha_2, \beta_1\alpha_2+\beta_2}$. In fact, for $\alpha_1=k_1/n^{p_1}$, $\beta_1=\ell_1/n^{q_1}$ and $\alpha_2=k_2/n^{p_2}$, $\beta_2=\ell_2/n^{q_2}$, we have
 \begin{align*}
a \stackrel{\varphi_{\alpha_1, \beta_1}}{\qquad\mapsto\qquad} t^{-p_1} a^{k_1} t^{p_1} \stackrel{\varphi_{\alpha_2, \beta_2}}{\qquad\mapsto\qquad} & \big( a^{\beta_2} t\big)^{-p_1} \big( a^{\alpha_2}\big)^{k_1} \big( a^{\beta_2} t\big)^{p_1}= \\ &= a^{\frac{n^{-p_1}-1}{n-1}\beta_2} t^{-p_1} a^{k_1\alpha_2} a^{\frac{n^{p_1}-1}{n-1}\beta_2} t^{p_1}= \\ &= a^{\frac{n^{-p_1}-1}{n-1}\beta_2} a^{n^{-p_1} k_1\alpha_2} a^{n^{-p_1}\frac{n^{p_1}-1}{n-1}\beta_2}= \\ &= a^{\frac{n^{-p_1}-1}{n-1}\beta_2+\alpha_1\alpha_2+\frac{1-n^{-p_1}}{n-1}\beta_2}= \\ &= a^{\alpha_1\alpha_2},
 \end{align*}
and
\begin{align*}
t \stackrel{\varphi_{\alpha_1, \beta_1}}{\qquad\mapsto\qquad} t^{-q_1} a^{\ell_1} t^{q_1+1} \stackrel{\varphi_{\alpha_2, \beta_2}}{\qquad\mapsto\qquad} & \big( a^{\beta_2} t\big)^{-q_1} \big( a^{\alpha_2} \big)^{\ell_1} \big( a^{\beta_2} t\big)^{q_1+1}= \\ &= a^{\frac{n^{-q_1}-1}{n-1}\beta_2} t^{-q_1} a^{\ell_1\alpha_2} a^{\frac{n^{q_1+1}-1}{n-1}\beta_2} t^{q_1+1}= \\ &= a^{\frac{n^{-q_1}-1}{n-1}\beta_2} a^{n^{-q_1}\ell_1\alpha_2} a^{n^{-q_1}\frac{n^{q_1+1}-1}{n-1}\beta_2} t= \\ &= a^{\frac{n^{-q_1}-1}{n-1}\beta_2+\beta_1\alpha_2+\frac{n-n^{-q_1}}{n-1}\beta_2} t= 
\\ &= a^{\beta_1\alpha_2+\beta_2} t.
 \end{align*}
Therefore, $\varphi_{\alpha_1,\beta_1} \circ \varphi_{\alpha_2,\beta_2} =\varphi_{\alpha_1\alpha_2, \beta_1\alpha_2+\beta_2}$, as we wanted to see. Note that this can be alternatively expressed by saying that $\varphi_{\alpha_1, \beta_1}\circ \varphi_{\alpha_2, \beta_2} =\varphi_{\alpha, \beta}$, where 
 $$
\left( \begin{array}{cc} \alpha & 0 \\ \beta & 1 \end{array}\right) = \left( \begin{array}{cc} \alpha_1 & 0 \\ \beta_1 & 1 \end{array}\right)\left( \begin{array}{cc} \alpha_2 & 0 \\ \beta_2 & 1 \end{array}\right)
 $$
in $\Mat_2(\ZZ[1/n])$. In particular, such an endomorphism $\varphi_{\alpha, \beta}$ is an automorphism if and only if $\left( \begin{smallmatrix} \alpha & 0 \\ \beta & 1 \end{smallmatrix} \right)\in \GL_2(\ZZ[1/n])$, which happens if and only if $\alpha=k/n^p$ is invertible in $\ZZ[1/n]$, i.e., if $k$ divides some power of $n$. This completes the proof. 
\end{proof}
In other words, we have proven that
\begin{equation}\label{eq: aut semi direct}
\Aut(BS(n,1))\simeq \left\{\left( \left.\begin{array}{cc} \alpha & 0 \\ \beta & 1 \end{array} \right) \,\, \right| \,\, \alpha\in \ZZ[1/n]^*,\, \beta\in \ZZ[1/n]\right\}\leqslant \GL_2(\ZZ[1/n]),
\end{equation}
with $\left( \begin{smallmatrix} \alpha & 0 \\ \beta & 1 \end{smallmatrix} \right)$ acting as $\varphi_{\alpha, \beta}$, namely, mapping any $a^{\nu}t^c\in BS(n,1)$ (with $\nu=m/n^r\in \ZZ[1/n]$, $m,r,c\in \ZZ$, $r\geq 0$) to 
 \begin{align}\label{eq: action}
\big( a^{\nu}t^c \big) \varphi_{\alpha,\beta} = \big( t^{-r}a^{m}t^{r+c}\big) \varphi_{\alpha, \beta} &= \big( a^{\beta}t\big)^{-r} a^{m\alpha} \big(a^{\beta}t\big)^{r+c}= \notag \\ &= a^{\frac{n^{-r}-1}{n-1}\beta}t^{-r} a^{m\alpha} a^{\frac{n^{r+c}-1}{n-1}\beta}t^{r+c}= \notag \\ &= a^{\frac{n^{-r}-1}{n-1}\beta} a^{m n^{-r}\alpha} a^{n^{-r}\frac{n^{r+c}-1}{n-1}\beta}t^{c}= \\ &= a^{\frac{n^{-r}-1}{n-1}\beta+ m n^{-r}\alpha + \frac{n^{c}-n^{-r}}{n-1}\beta}t^{c}= \notag \\ &= a^{\nu \alpha + \frac{n^{c}-1}{n-1}\beta}t^{c}. \notag
 \end{align}

\subsection{The Conjugacy Problem in $BS(n,1)$}\label{sec: CP}

Let us observe the role of inner automorphisms. Left conjugation by an arbitrary element $a^{\beta}t^r\in BS(n,1)$, $\beta\in \ZZ[1/n]$, $r\in \ZZ$, acts as
 \begin{align*}
a\quad \mapsto \quad & \big( a^{\beta} t^r \big)\, a \,\big( a^{\beta} t^r \big)^{-1}= a^{\beta}t^r a t^{-r} a^{-\beta} =a^{\beta}a^{n^r} a^{-\beta} =a^{n^r}, \\ t \quad \mapsto \quad & \big( a^{\beta} t^r \big)\, t \,\big( a^{\beta} t^r \big)^{-1}= a^{\beta}t a^{-\beta} =a^{\beta}a^{-n\beta} t=a^{(1-n)\beta}t, 
 \end{align*}
That is, left conjugation by $a^{\beta}t^r$ equals $\varphi_{n^r,(1-n)\beta}$. Therefore, 
 $$
\Inn (BS(n,1))=\left\{ \varphi_{\alpha, \beta} \mid \alpha =n^r\, (r\in \ZZ),\, \beta\in (n-1)\ZZ[\frac{1}{n}] \right\} \unlhd \Aut(BS(n,1)).
 $$

As a straightforward consequence, we can deduce the following solution to the Conjugacy Problem for $BS(n,1)$, a well-know classical result (see~\cite{Anshel-Stebe}):

\begin{prop}\label{prop: CP}
The Conjugacy Problem is solvable in $BS(n,1)$.
\end{prop}

\begin{proof}
Let $a^{\nu_1}t^{c_1},\, a^{\nu_2}t^{c_2}\in BS(n,1)$ be two arbitrary elements, where $\nu_i =k_i/n^{p_i}\in \ZZ[1/n]$ with $k_i, p_i, c_i\in \ZZ$, $p_i\geq 0$, $i=1,2$. Inverting them, if necessary, we can assume $c_1\geq 0$. By the previous analysis and equation~\eqref{eq: action}, left conjugation by $a^{\beta}t^r$ acts as 
$\varphi_{n^r,(1-n)\beta}$. Therefore, $a^{\nu_1}t^{c_1}$ and $a^{\nu_2}t^{c_2}$ are conjugated to each other if and only if $c_1=c_2$ (let us call it just $c\geq 0$) and there exists $\beta\in \ZZ[1/n]$ and $r\in \ZZ$ such that 
 $$
\nu_2 =\nu_1 n^r +\frac{n^c-1}{n-1}(1-n)\beta =\nu_1 n^r -(n^c-1)\beta.
 $$
This happens if and only if there exists $\beta=\ell/n^q\in \ZZ[1/n]$, $\ell,q\in \ZZ$, $q\geq 0$, and $r\in \ZZ$, such that 
 $$
\frac{k_2}{n^{p_2}}= \frac{k_1 n^r}{n^{p_1}}-(n^c-1)\frac{\ell}{n^{q}}=\frac{k_1n^{r+q}-(n^c-1)\ell n^{p_1}}{n^{p_1+q}}
 $$
 $$
k_2 n^{p_1+q}=k_1n^{p_2+r+q}-(n^c-1)\ell n^{p_1+p_2}.
 $$
Taking $q=p_1+p_2$, canceling out $n^{p_1+p_2}$ from all this integral equation, and reducing it modulo $n^c-1$, this happens if and only if there exists $r'=0,1,\ldots ,c-1$ such that 
 $$
k_2 \equiv k_1 n^{r'} \pmod{n^c-1}.
 $$
This can be easily decided, completing the proof. 
%
\end{proof}

\section{The twisted conjugacy problem in $BS(n,1)$}\label{sec: TCP}

Our goal in this section is to solve the Twisted Conjugacy Problem for the group $BS(n,1)$. First, we record the following general and elementary fact, which is true in an arbitrary group $G$. The proof is direct from the definitions.

\begin{lem} \label{lem: twisted autos}
Let $G$ be a group, $u,v,w \in G$, $\varphi, \psi \in \Aut(G)$ and $\gamma_w \in \Inn(G)$. Then the following are equivalent:
 \begin{itemize}
\item[(i)] $u \sim_\phi v$,
\item[(ii)] $u \psi \sim_{\psi^{-1}\varphi\psi} v \psi$,
\item[(iii)] $wu \sim_{\varphi\gamma_{w^{-1}}} wv$, \item[(iv)] $uw \sim_{\gamma_{w^{-1}}\varphi} vw$. \hfill \qed
 \end{itemize}
\end{lem}

\begin{thm}\label{thm: main TCP}
The Twisted Conjugacy Problem is solvable in $BS(n,1)$.
\end{thm}

\begin{proof}
Suppose we are given a fixed automorphism $\varphi_{\alpha, \beta}\in \Aut(BS(n,1))$, 
where $\alpha=k/n^p \in \ZZ[1/n]^*, \beta=\ell/n^q\in \ZZ[1/n]$ with $k,\ell, p,q\in \ZZ$, $p,q\geq 0$. Suppose we are also given two elements $u=a^{\nu_1}t^{c_1},\, v=a^{\nu_2}t^{c_2}\in BS(n,1)$, where $\nu_i =m_i/n^{r_i}\in \ZZ[1/n]$ with $m_i, r_i, c_i\in \ZZ$, $r_i\geq 0$, $i=1,2$. We have to decide whether $a^{\nu_1}t^{c_1}$ and $a^{\nu_2}t^{c_2}$ are $\varphi_{\alpha, \beta}$-twisted conjugated to each other, i.e., whether $u\sim_{\varphi_{\alpha, \beta}} v$. Observe that, again, $c_1=c_2$ (call it just $c$) is an obvious necessary condition. 

Applying Lemma~\ref{lem: twisted autos} (i)$\Leftrightarrow$(iv) with $w=t^{-c}$, we can assume $c_1=c_2=0$, i.e., $a^{\nu_i}t^{c_i}=t^{-r_i}a^{m_i}t^{r_i}$, $i=1,2$. Then, applying Lemma~\ref{lem: twisted autos} (i)$\Leftrightarrow$(ii) with $\psi=\gamma_{t^{-r_1}}$, we can assume $r_1=0$. In other words, we are reduced to inputs of the form $u=a^{m_1},\, v=a^{\nu_2}=t^{-r_2}a^{m_2}t^{r_2}\in BS(n,1)$.

Writing the possible $\varphi_{\alpha, \beta}$-twisted conjugator as $g=a^{\nu}t^d$, $\nu\in \ZZ[1/n]$, $d\in \ZZ$, and using equations~\eqref{eq: action} and~\eqref{eq: inverse}, we have 
 $$
g\varphi_{\alpha,\beta} =a^{\nu \alpha + \frac{n^{d}-1}{n-1}\beta}t^{d},
 $$
 $$
\big( g\varphi_{\alpha,\beta} \big)^{-1} =a^{-n^{-d}(\nu \alpha + \frac{n^{d}-1}{n-1}\beta)}t^{-d},
 $$
and 
 $$
\big( g\varphi_{\alpha,\beta} \big)^{-1} ug= a^{-n^{-d}(\nu \alpha + \frac{n^{d}-1}{n-1}\beta)} t^{-d} a^{m_1} a^{\nu}t^d= a^{-n^{-d}\nu \alpha + \frac{n^{-d}-1}{n-1}\beta} a^{m_1 n^{-d}} a^{\nu n^{-d}}.
 $$
So, $a^{m_1}\sim_{\varphi_{\alpha, \beta}} a^{\nu_2}$ if and only if 
 $$
\frac{m_2}{n^{r_2}} n^d =-\nu \alpha +\frac{1-n^{d}}{n-1}\beta +m_1 +\nu, 
 $$
 $$
\Leftrightarrow (n-1)m_2 n^d +(n-1)n^{r_2}\nu (\alpha -1) +(n^{d}-1)n^{r_2}\beta =(n-1)n^{r_2}m_1, 
 $$
for some $\nu\in \ZZ[1/n]$ and $d\in \ZZ$. Writing $\nu =y/n^x$, $x,y\in \ZZ$, $x\geq 0$, and applying the change of variable $\ZZ \ni z=d+x$, this is equivalent to the integral equation  
 $$
(n-1)m_2 n^d +(n-1)n^{r_2}\frac{y}{n^x} \frac{k-n^p}{n^p} +(n^{d}-1)n^{r_2}\frac{\ell}{n^q} =(n-1)n^{r_2}m_1, 
 $$
 $$
\Rightarrow (n-1)m_2 n^{d+x+p+q} +(n-1)y(k-n^p)n^{r_2+q} +(n^{d}-1)\ell n^{r_2+p+x} =(n-1)m_1 n^{r_2+p+q+x}, 
 $$
 $$
\Rightarrow \Big( \ell n^{r_2+p}+(n-1)m_1 n^{r_2+p+q}\Big) n^x +\Big( (n-1)(n^p-k)n^{r_2+q} \Big) y =\Big( (n-1)m_2 n^{p+q} +\ell n^{r_2+p}\Big) n^{z}.
 $$
From the data $k, p, \ell, q, m_1, m_2, r_2, n\in \ZZ$, $p,q,r_2,n\geq 0$, compute the three integers 
 $$
A=\ell n^{r_2+p}+(n-1)m_1 n^{r_2+p+q},\qquad B=(n-1)(n^p-k)n^{r_2+q}, \qquad C=(n-1)m_2 n^{p+q}+\ell n^{r_2+p}.
 $$
Summarizing, $A,B,C\in \ZZ$, and $u$ and $v$ are $\varphi_{\alpha, \beta}$-twisted conjugated to each other if and only if the equation 
 \begin{equation}\label{eq: equation}
An^x +By=Cn^z
 \end{equation}
has an integral solution, for the unknowns $x,y,z\in \ZZ$, with $x\geq 0$. Write $C=n^s C'$ with $s\geq 0$ and $C'$ not being multiple of $n$, and note that all possible solutions to equation~\eqref{eq: equation} have $z\geq -s$. So, with the change of variable $\ZZ \ni z'=z+s$, \eqref{eq: equation} is equivalent to 
 \begin{equation}\label{eq: equation2}
An^x +By=C'n^{z'},
 \end{equation}
for the unknowns $x,y,z'\in \ZZ$, with $x,z'\geq 0$. Let us distinguish three cases. 

\textbf{Case 1: $B=\pm 1$}. Obviously, equation~\eqref{eq: equation2} always has solutions.  

\textbf{Case 2: $B=0$}. Our equation becomes $An^x=C' n^{z'}$, which has the desired solution if and only if either $A=C'=0$, or both $A,C'\neq 0$ and $A/C$ is a rational number being a (positive or negative) power of $n$. This can be easily checked by looking at the prime factorizations of $A$, $C'$, and $n$.

\textbf{Case 3: $B\neq 0$}. Now, equation~\eqref{eq: equation2} is equivalent to 
 $$
An^x \equiv C'n^{z'} \pmod{B},
 $$
for the unknowns $x,z'\in \ZZ$, with $x,z'\geq 0$. Compute the finite set $\mathcal{N}=\{1,n,n^2, \ldots \}\subseteq \ZZ/B\ZZ$ (by computing the successive powers of $n$ until obtaining the first repetition modulo $B$). Clearly, our equation admits a solution of and only if the subsets $A\mathcal{N}$ and $C'\mathcal{N}$ from $\ZZ/B\ZZ$ intersect non-trivially. 

This completes the proof.
\end{proof}

\section{Orbit Decidability for the full $\Aut(BS(n,1))$}\label{sec: OD}

In this section we are going to solve the Orbit Decidability Problem for the whole automorphism group $\Aut(BS(n,1))$, i.e., given two elements $u,v \in BS(n,1)$ we will decide whether there exists $\varphi_{\alpha, \beta} \in \Aut(BS(n,1))$ such that $u\varphi_{\alpha, \beta}=v$. Before, we need to remind a folklore algorithmic result about the ring $\big( \ZZ[1/n], +, \cdot \big)$. (Note that, till now, we have considered $\ZZ[1,n]$ just as an additive group, and $\ZZ[1/n]^*$ as the (multiplicative) group of units of the ring $\big( \ZZ[1/n], +, \cdot \big)$.) 

It is well known that, as a ring, $\ZZ[1/n]$ is an Euclidean domain; in particular, (i) we have an easy algorithm to compute greatest common divisors of elements in $\ZZ[1/n]$; and (ii) $\ZZ[1/n]$ is a principal ideal domain. See~\cite{Dummit-Foote} for a general basic reference. 

\begin{prop}\label{prop: coset of ideal}
There is an algorithm which, given $\alpha, \delta\in \ZZ[1/n]$ decides whether the coset $\alpha +\delta\ZZ[1/n]$ contains an invertible element. 
\end{prop}

\begin{proof}
Note that generators of (principal) ideals in $\ZZ[1/n]$ work up to products by units of the ring, So, without loss of generality, we can assume $\delta\in \ZZ$. The coset at play is then the subset 
 $$
\alpha +\delta\ZZ[1/n] =\{ \alpha +\delta\beta \mid \beta\in \ZZ[1/n]\}\subseteq \ZZ[1/n].
 $$
So, given $\alpha=k/n^r\in \ZZ[1/n]$, $k,r\in \ZZ$, $r\geq 0$, and given $\delta\in \ZZ$,  we have to decide whether there exists $x,z\in \ZZ$, $z\geq 0$, such that 
 $$
\frac{k}{n^r}+\delta \frac{x}{n^z} =\frac{kn^z+\delta x n^r}{n^{r+z}}
 $$
or, equivalently, $kn^z+\delta n^r x$ belongs to $\ZZ[1/n]^*$. In other words, we have to decide whether there exists $x,z\in \ZZ$, $z\geq 0$, such that the integer $kn^z+\delta n^r x$ factorizes involving only primes from the set $\{p_1, \ldots ,p_m\}$ (those appearing in the prime factorization of $n=p_1^{e_1}\cdots p_m^{e_m}$).  

In order to decide this, compute the finite sets
 $$
\mathcal{N}=\{1, n, n^2, \ldots\}\subseteq \ZZ/\delta \ZZ,
 $$
 $$
\mathcal{P}_i=\{1, p_i, p_i^2, \ldots \}\subseteq \ZZ/\delta \ZZ,
 $$
for $i=1,\ldots ,m$ (note that, on each  one, there are infinitely many powers to compute, but they take only finitely many values modulo $\delta$, so we only have to compute until getting the first repetition). Compute also the finite set 
 $$
\mathcal{P}=\{y_1\cdots y_m \mid y_1\in \mathcal{P}_1, \ldots , y_m\in \mathcal{P}_m \}\subseteq \ZZ/\delta \ZZ.
 $$
Let us distinguish two cases, $z\geq r$ and $z\leq r$. 
\begin{itemize}
\item There exists $x,z\in \ZZ$, $z\geq r$, such that $kn^z+\delta n^r x$ is invertible in $\ZZ[1/n]$ if and only if there exists $x,z'\in \ZZ$, $z'\geq 0$, such that $kn^{z'}+\delta x\in \ZZ[1/n]^*$, and this happens if and only if $\mathcal{P}$ and $k\mathcal{N}$ intersect non-trivially as subsets of $\ZZ/\delta \ZZ$. 
\item There exists $x,z\in \ZZ$, $z\leq r$, such that $kn^z+\delta n^r x$ is invertible in $\ZZ[1/n]$ if and only if there exists $x\in \ZZ$ such that at least one of $k+\delta xn^r,\, \ldots ,k+\delta xn^0$ belongs to $\ZZ[1/n]^*$; which is equivalent to the existence of $x\in \ZZ$ such that $k+\delta x\in \ZZ[1/n]^*$. And this happens if and only if $k\in \mathcal{P} \pmod{\delta}$. 
\end{itemize} 
These two conditions are clearly decidable with elementary algorithms. This completes the proof. 
\end{proof}

\begin{thm}\label{thm: main 1 OD}
The full automorphism group $A=\Aut(BS(n,1))$ is orbit decidable.
\end{thm}

\begin{proof}
We first recall from~\eqref{eq: aut semi direct} and~\eqref{eq: action} that $\Aut(BS(n,1))=\{ \varphi_{\alpha, \beta} \mid \alpha\in \ZZ[1/n]^*,\, \beta\in \ZZ[1/n]\}$ and the action is described by the equation $(a^\nu t^c)\varphi_{\alpha, \beta} = a^{\nu \alpha + \frac{n^{c}-1}{n-1}\beta}t^{c}$.

Let $u=a^{\nu_1}t^{c_1},\, v=a^{\nu_2}t^{c_2}\in BS(n,1)$ be two given elements, where $\nu_i =k_i/n^{p_i}\in \ZZ[1/n]$ with $k_i, p_i, c_i\in \ZZ$, $p_i\geq 0$, $i=1,2$. We have to decide whether $u$ and $v$ belong to the same orbit, i.e., $u\varphi_{\alpha, \beta} =v$ for some $\alpha\in \ZZ[1/n]^*$ and some $\beta\in \ZZ[1/n]$. As observed above, $c_1 =c_2$ is a necessary condition; let us just denote it by $c$. 

Consider the equation 
 $$
\nu_2 =\nu_1 \alpha +\mu\beta, 
 $$
with unknowns $\alpha \in \ZZ[1/n]^*$ and $\beta \in \ZZ[1/n]$, where $\mu =\frac{n^c -1}{n-1}\in \ZZ \leq\ZZ[1/n]$. Using the fact that $\ZZ[1/n]$ is an Euclidean domain, compute $\gcd(\nu_1, \mu)$; if $\nu_2$ is not multiple of it then, clearly, the above equation has no solution and we are done. So, we can assume it is and, simplifying $\gcd(\nu_1, \mu)$, we are reduced to consider the equation  
 \begin{equation}\label{eq: bezout}
\nu'_2 =\nu'_1 \alpha +\mu'\beta, 
 \end{equation}
where, now, $\gcd(\nu'_1, \mu')=1$. Using Bezout's identity, we can effectively compute a particular solution $\alpha_0,\beta_0\in \ZZ[1/n]$ to it, 
 $$
\nu'_2 =\nu'_1 \alpha_0 +\mu'\beta_0. 
 $$
If, by chance, $\alpha_0$ is invertible, we conclude that $u$ and $v$ belong to the same orbit. Otherwise, we need to run over \emph{all} possible alternative solutions to~\eqref{eq: bezout} and check for invertibility of the corresponding $\alpha$. Note that, for any other possible solution to~\eqref{eq: bezout}, say $\nu'_2 =\nu'_1 \alpha +\mu'\beta$, we have $\nu'_1 (\alpha-\alpha_0 )+\mu' (\beta-\beta_0 )=0$ and so (since $\gcd(\nu'_1, \mu')=1$), $\alpha-\alpha_0 =\lambda \mu'$ and $\beta-\beta_0 =-\lambda \nu_1'$, for some $\lambda\in \ZZ[1/n]$. This means that \emph{all} solutions to ~\eqref{eq: bezout} are of the form 
 $$
\left. \begin{array}{rcl} \alpha & = & \alpha_0+\lambda \mu' \\ \beta & = & \beta_0 -\lambda \nu_1' \end{array} \right\},
 $$
for $\lambda\in \ZZ[1/n]$. It remains to know whether $\alpha_0+\lambda \mu'$ is invertible for some $\lambda\in \ZZ[1/n]$. This can be decided by Proposition~\ref{prop: coset of ideal}. 
\end{proof}

Theorem~\ref{thm: main 1 OD} is the analogous to $BS(n,1)$ of the classical Whitehead's result for free groups. As mentioned above, the following is an immediate consequence using Theorem~\ref{thm: TCP to CP}. 

\begin{cor}
Let $H$ be a torsion-free hyperbolic group and $\varphi_1,\ldots ,\varphi_m\in \Aut(BS(n,1))$ be automorphisms such that $\gen{\varphi_1,\ldots ,\varphi_m} =\Aut(BS(n,1))$. Then, the Conjugacy Problem is solvable in any group of the form $BS(n,1)\rtimes_{\varphi_1, \ldots ,\varphi_m} H$. \qed
\end{cor}

\subsection*{Acknowledgements:}

The first and second authors gratefully acknowledge the Chennai Mathematical Institute (CMI), where this project was initiated. The first author thanks CMI and the National Board for Higher Mathematics (NBHM) for the travel grant that enabled her visit to the Universitat Polit\`ecnica de Catalunya (UPC). She is grateful to the INSPIRE Faculty Fellowship of the Department of Science and Technology (DST), Government of India. The second author thanks the support from UPC through a Margarita Salas grant. She also expresses her gratitude for the generous hospitality received from Harish-Chandra Research Institute. The second and third authors acknowledge support from the Spanish Agencia Estatal de Investigaci\'on through grant PID2021-126851NB-100 (AEI/ FEDER, UE).

\end{document}